\documentclass[a4paper,11pt]{amsart}
\usepackage[utf8]{inputenc}
\usepackage{amsmath,amssymb,amsfonts}
\usepackage{epsfig}
\usepackage{wasysym}
\usepackage{calrsfs}
\usepackage{graphicx,color}
\theoremstyle{definition}
\newtheorem{dfn}{Definition}[section]
  
\theoremstyle{plain}
\newtheorem{thm}{Theorem}[section]
\newtheorem{pro}{Proposition}[section]
\newtheorem{cor}{Corollary}[section]

\theoremstyle{definition}
\newtheorem{rem}{Remark}[section]
\newtheorem{exa}{Example}[section]

\newcommand{\N}{\mathbb{N}}
\newcommand{\R}{\mathbb{R}}
\newcommand{\E}{\mathbb{E}}
\renewcommand{\P}{\mathbb{P}}

\newcommand{\G}{\mathbf{G}}
\newcommand{\g}{\mathbf{g}}
\newcommand{\y}{\mathbf{y}}
\newcommand{\0}{\mathbf{0}}
\newcommand{\lla}{\langle\!\langle\!\langle}
\newcommand{\rra}{\rangle\!\rangle\!\rangle}

\newcommand{\Cov}{\mathbb{C}\mathrm{ov}}

\newcommand{\e}{\mathrm{e}}
\newcommand{\la}{\langle}
\newcommand{\ra}{\rangle}
\renewcommand{\H}{\mathcal{H}}
\newcommand{\1}{\mathbf{1}}
\renewcommand{\d}{\mathrm{d}}

\numberwithin{equation}{section}

\begin{document}
\title[Fredholm Representation]
{Stochastic Analysis of Gaussian Processes via Fredholm Representation}

\author[Sottinen]{Tommi Sottinen}
\address{Tommi Sottinen\\ Department of Mathematics and Statistics, University of Vaasa, P.O. Box 700, FIN-65101 Vaasa, FINLAND}
\email{tommi.sottinen@iki.fi}

\author[Viitasaari]{Lauri Viitasaari}
\address{Lauri Viitasaari\\ Department of Mathematics and System Analysis, Aalto University School of Science, Helsinki\\
P.O. Box 11100, FIN-00076 Aalto,  FINLAND}
\email{lauri.viitasaari@aalto.fi}

\begin{abstract}
We show that every separable Gaussian process with integrable variance function admits a Fredholm representation with respect to a Brownian motion. We extend the Fredholm representation to a transfer principle and develop stochastic analysis by using it. We show the convenience of the Fredholm representation by giving applications to equivalence in law, bridges, series expansions, stochastic differential equations and maximum likelihood estimations.
\end{abstract}

\date{\today}



\thanks{Lauri Viitasaari was partially funded by Emil Aaltonen Foundation.
Tommi Sottinen was partially funded by the Finnish Cultural Foundation (National Foundations' Professor Pool).} 

\maketitle

%%%%%%%%%%%%%%%%%%%%%%%%%%%%%%%%%%%%%%%%%%%%%%%%%%%%%%%%%%%%%%%%%%%%%%%%%%%%%%%
%%%%%%%%%%%%%%%%%%%%%%%%%%%%%%%%%%%%%%%%%%%%%%%%%%%%%%%%%%%%%%%%%%%%%%%%%%%%%%%
%%%%%%%%%%%%%%%%%%%%%%%%%%%%%%%%%%%%%%%%%%%%%%%%%%%%%%%%%%%%%%%%%%%%%%%%%%%%%%%
\section{Introduction}

The stochastic analysis of Gaussian processes that are not semimartingales is challenging. One way to overcome the challenge is to represent the Gaussian process under consideration, $X$ say, in terms of a Brownian motion and then develop a transfer principle so that that the stochastic analysis can be done in the ``Brownian level'' and then transfered back into the level of $X$. 

One of the most studied representation in terms of a Brownian motion is the so-called Volterra representation. A \emph{Gaussian Volterra process} is a process that can be represented as
\begin{equation}\label{eq:intro-volterra}
X_t = \int_0^t K(t,s)\, \d W_s, \qquad t\in [0,T],
\end{equation}
where $W$ is a Brownian motion and $K\in L^2([0,T]^2)$. Here the integration goes only upto $t$, hence the name ``Volterra''.  This Volterra nature is very convenient: it means that the filtration of $X$ is included in the filtration of the underlying Brownian motion $W$.
Gaussian Volterra processes and their stochastic analysis has been studied, e.g., in \cite{alos-mazet-nualart} and \cite{baudoin-nualart}, just to mention few.  Apparently, the most famous Gaussian process admitting Volterra representation is the fractional Brownian motion and its stochastic analysis indeed has been developed mostly by using its Volterra representation, see e.g. the monographs \cite{oksendal} and \cite{mishura} and references therein.

In discrete finite time the Volterra representation \eqref{eq:intro-volterra} is nothing but the Cholesky lower-triangular factorization of the covariance of $X$, and hence every Gaussian process is a Volterra process.  In continuous time this is not true, see Example \ref{exa:degenerate} in Section \ref{sect:fredholm}.  

There is a more general representation than \eqref{eq:intro-volterra} by Hida, see \cite[Theorem 4.1]{hida-hitsuda}. However, this Hida representation includes possibly infinite number of Brownian motions. Consequently, it seems very difficult to apply the Hida representation to build a transfer principle needed by stochastic analysis.  Moreover, the Hida representation is not quite general. Indeed, it requires, among other things, that the Gaussian process is purely non-deterministic.  The Fredholm representation \eqref{eq:intro-fredholm} below does not require pure non-determinism. Our Example \ref{exa:degenerate} in Section \ref{sect:fredholm}, that admits a Fredholm representation, does not admit a Hida representation, and the reason is the lack of pure non-determinism.

The problem with the Volterra representation \eqref{eq:intro-volterra} is the Volterra nature of the kernel $K$, as far as generality is concerned.  Indeed, if one considers Fredholm kernels, i.e., kernels where the integration is over the entire interval $[0,T]$ under consideration, one obtains generality.  A \emph{Gaussian Fredholm process} is a process that admits the Fredholm representation    
\begin{equation}\label{eq:intro-fredholm}
X_t = \int_0^T K_T(t,s)\,\d W_s, \qquad t\in [0,T],
\end{equation}
where $W$ is a Brownian motion and $K_T\in L^2([0,T]^2)$. In this paper we show that \emph{every} separable Gaussian process with integrable variance function admits the representation \eqref{eq:intro-fredholm}.  The price we have to pay for this generality is twofold:
\begin{enumerate}
\item 
The process $X$ is generated, in principle, from the entire path of the underlying Brownian motion $W$.  Consequently, $X$ and $W$ do not necessarily generate the same filtration.  This is unfortunate in many applications. 
\item
In general the kernel $K_T$ depends on $T$ even if the covariance $R$ does not, and consequently the derived operators also depend on $T$. This is why we use the cumbersome notation of explicitly stating out the dependence when there is one. In stochastic analysis this dependence on $T$ seems to be a minor inconvenience, however. Indeed, even in the Volterra case as examined, e.g., by Al\`os, Mazet and Nualart \cite{alos-mazet-nualart}, one cannot avoid the dependence on $T$ in the transfer principle. 
Of course, for statistics, where one would like to let $T$ tend to infinity, this is a major inconvenience.
\end{enumerate}

Let us note that the Fredholm representation has already been used, without proof, in \cite{azmoodeh-et-al}, where the H\"older continuity of Gaussian processes was studied.

\medskip

Let us mention a few papers that study stochastic analysis of Gaussian processes here. Indeed, several different approaches have been proposed in the literature. In particular, fractional Brownian motion has been a subject of active study (see the monographs \cite{oksendal} and \cite{mishura} and references therein). More general Gaussian processes have been studied in the already mentioned work by Al\`os, Mazet and Nualart \cite{alos-mazet-nualart}. They considered Gaussian Volterra processes where the kernel satisfies certain technical conditions. In particular, their results cover fractional Brownian motion with Hurst parameter $H>\frac14$. Later Cheridito and Nualart \cite{che-nua} introduced an approach based on the covariance function itself rather than to the Volterra kernel $K$. Kruk et al. \cite{kruk-et-al} developed stochastic calculus for processes having finite $2$-planar variation, especially covering fractional Brownian motion $H\geq\frac12$. Moreover, Kruk and Russo \cite{kruk-russo} extended the approach to cover singular covariances, hence covering fractional Brownian motion $H<\frac12$. Furthermore, Mocioalca and Viens \cite{mo-viens} studied processes which are close to processes with stationary increments. More precisely, their results cover cases where $\E(X_t - X_s)^2 \sim \gamma^2(|t-s|)$ where $\gamma$ satisfies some minimal regularity conditions. In particular, their results cover some processes which are not even continuous. Finally, the latest development we are aware of is a paper by Lei and Nualart \cite{lei-nualart} who developed stochastic calculus for processes having absolute continuous covariance by using extended domain of the divergence introduced in \cite{kruk-russo}. Finally, we would like to mention Lebovits \cite{lebovits} who used the S-transform approach and obtained similar results to ours, albeit his notion of integral is not elementary as ours.

The results presented in this paper gives unified approach to stochastic calculus for Gaussian processes and only integrability of the variance function is required. In particular, our results cover processes that are not continuous. 

\medskip

The paper is organized as follows: 

Section \ref{sect:preliminaries} contains some preliminaries on Gaussian processes and isonormal Gaussian processes and related Hilbert spaces.

Section \ref{sect:fredholm} provides the proof of the main theorem of the paper: the Fredholm representation. 

In Section \ref{sect:sa} we extend the Fredholm representation to a transfer principle in three contexts of growing generality:  First we prove the transfer principle for Wiener integrals in Subsection \ref{subs:wi}, then we use the transfer principle to define the multiple Wiener integral in Subsection \ref{subs:mwi}, and finally, in Subsection \ref{subs:mc} we prove the transfer principle for Malliavin calculus, thus showing that the definition of multiple Wiener integral via the transfer principle done in Subsection \ref{subs:mwi} is consistent with the classical definitions involving Brownian motion or other Gaussian martingales. Indeed, classically one defines the multiple Wiener integrals by either building an isometry with removed diagonals or by spanning higher chaoses by using the Hermite polynomials.  In the general Gaussian case one cannot of course remove the diagonals, but the Hermite polynomial approach is still valid.  We show that this approach is equivalent to the transfer principle. In Subsection \ref{subs:mc} we also prove an It\^o formula for general Gaussian processes and in Subsection \ref{subs:ed} we extend the It\^o formula even further by using the technique of extended domain in the spirit of \cite{che-nua}. This It\^o formula is, as far as we know, the most general version for Gaussian processes existing in the literature so far. 

Finally, in Section \ref{sect:appl} we show the power of the transfer principle in some applications. 
In Subsection \ref{subs:el} the transfer principle is applied to the question of equivalence of law of general Gaussian processes.
In subsection \ref{subs:bridges} we show how one can construct net canonical-type representation for generalized Gaussian bridges, i.e., for the Gaussian process that is conditioned by multiple linear functionals of its path. 
In Subsection \ref{subs:se} the transfer principle is used to provide series expansions for general Gaussian processes. 
    
%%%%%%%%%%%%%%%%%%%%%%%%%%%%%%%%%%%%%%%%%%%%%%%%%%%%%%%%%%%%%%%%%%%%%%%%%%%%%%%
%%%%%%%%%%%%%%%%%%%%%%%%%%%%%%%%%%%%%%%%%%%%%%%%%%%%%%%%%%%%%%%%%%%%%%%%%%%%%%%
\section{Preliminaries}\label{sect:preliminaries}

Our general setting is as follows:
Let $T>0$ be a fixed finite time-horizon and let $X=(X_t)_{t\in [0,T]}$ be a Gaussian process with covariance $R$ that may or may not depend on $T$. Without loss of any interesting generality we assume that $X$ is centered. We also make the very weak assumption that $X$ is \emph{separable} in the sense of following definition.

\begin{dfn}[Separability]\label{dfn:separable}
The Gaussian process $X$ is separable if the Hilbert space $L^2(\Omega,\sigma(X),\P)$ is separable.
\end{dfn}

\begin{exa}
If the covariance $R$ is continuous, then $X$ is separable.  In particular, all continuous Gaussian processes are separable.
\end{exa}

\begin{dfn}[Associated operator]\label{dfn:associated}
For a kernel $\Gamma\in L^2([0,T]^2)$ we associate an operator on $L^2([0,T])$, also denoted by $\Gamma$, as
$$
\Gamma f(t) = \int_0^T f(s)\Gamma(t,s)\, \d s. 
$$ 
\end{dfn}

\begin{dfn}[Isonormal process]\label{dfn:isonormal}
The isonormal process associated with $X$, also denoted by $X$, is the Gaussian family $(X(h), h\in\H_T)$, where the Hilbert space $\H_T=\H_T(R)$ is generated by the covariance $R$ as follows:
\begin{enumerate}
\item indicators $\1_t := \1_{[0,t)}$, $t\le T$, belong to $\H_T$.
\item $\H_T$ is endowed with the inner product $\la \1_t,\1_s\ra_{\H_T} := R(t,s)$.  
\end{enumerate}
\end{dfn}

Definition \ref{dfn:isonormal} states that $X(h)$ is the image of $h\in\H_T$ in the isometry that extends the relation
$$
X\left(\1_t\right) := X_t
$$
linearly. Consequently, we can define:

\begin{dfn}[Wiener integral]\label{dfn:wi}
$X(h)$ is the \emph{Wiener integral} of the element $h\in\H_T$ with respect to $X$. We shall also denote
$$
\int_0^T h(t)\, \d X_t := X(h).
$$   
\end{dfn}

\begin{rem}
Eventually, all the following will mean the same:
$$
X(h) = \int_0^T h(t)\, \d X_t = \int_0^T h(t)\, \delta X_t
= I_{T,1}(h) = I_T(h).
$$
\end{rem}

\begin{rem}
The Hilbert space $\H_T$ is separable if and only if $X$ is separable.
\end{rem}

\begin{rem}
Due to the completion under the inner product $\la\cdot,\cdot\ra_{\H_T}$ it may happen that the space $\H_T$ is not a space of functions, but contains distributions, cf. \cite{pipiras-taqqu} for the case of fractional Brownian motions with Hurst index bigger than half.  
\end{rem}

\begin{dfn}
The function space $\H_T^0\subset \H_T$ is the space of functions that can be approximated by step-functions on $[0,T]$ in the inner product $\la\cdot,\cdot\ra_{\H_T}$. 
\end{dfn}

\begin{exa}
If the covariance $R$ is of bounded variation, then $\H_T^0$ is the space of functions $f$ satisfying
$$
\int_0^T\!\int_0^T |f(t)f(s)|\, |R|(\d s, \d t) < \infty.
$$
\end{exa}

\begin{rem}
Note that it may be that $f\in\H_T^0$ but for some $T'<T$ we have $f\1_{T'}\not\in \H_{T'}^0$, cf. \cite{bender-elliott} for an example with fractional Brownian motion with Hurst index less than half.  For this reason we keep the notation $\H_T$ instead of simply writing $\H$. For the same reason we include the dependence of $T$ whenever there is one.
\end{rem}

%%%%%%%%%%%%%%%%%%%%%%%%%%%%%%%%%%%%%%%%%%%%%%%%%%%%%%%%%%%%%%%%%%%%%%%%%%%%%%%    
%%%%%%%%%%%%%%%%%%%%%%%%%%%%%%%%%%%%%%%%%%%%%%%%%%%%%%%%%%%%%%%%%%%%%%%%%%%%%%%
\section{Fredholm Representation}\label{sect:fredholm}

\begin{thm}[Fredholm representation]\label{thm:fredholm}
Let $X=(X_t)_{t\in [0,T]}$ be a separable centered Gaussian process. Then there exists a kernel $K_T\in L^2([0,T]^2)$ and a Brownian motion $W=(W_t)_{t\ge0}$, independent of $T$, such that
\begin{equation}\label{eq:fredholm}
X_t = \int_0^T K_T(t,s)\, \d W_s
\end{equation}
if and only if the covariance $R$ of $X$ satisfies the trace condition
\begin{equation}\label{eq:trace-condition}
\int_0^T R(t,t)\, \d t < \infty.
\end{equation}

The representation \eqref{eq:fredholm} is unique in the sense that any other representation with kernel $\tilde K_T$, say, is connected to \eqref{eq:fredholm} by a unitary operator $U$ on $L^2([0,T])$ such that $\tilde K_T = U K_T$.
Moreover, one may assume that $K_T$ is symmetric.
\end{thm}

\begin{proof}%[Proof of Theorem \ref{thm:fredholm}]
Let us first remark that \eqref{eq:trace-condition} is precisely what we need to invoke the Mercer's theorem and take square root in the resulting expansion.

Now, by the Mercer's theorem we can expand the covariance function $R$ on $[0,T]^2$ as
\begin{equation}\label{eq:Mercer}
R(t,s) = \sum_{i=1}^\infty \lambda_i^T e_i^T(t) e_i^T(s),
\end{equation}
where $(\lambda_i^T)_{i=1}^\infty$ and $(e_i^T)_{i=1}^\infty$ are the eigenvalues and the eigenfunctions of the covariance operator
$$
R_Tf (t) = \int_0^T f(s) R(t,s)\, \d s.
$$
Moreover, $(e_i^T)_{i=1}^\infty$ is an orthonormal system on $L^2([0,T])$.

Now, $R_T$, being a covariance operator, admits a square root operator $K_T$ defined by the relation 
\begin{equation}\label{eq:square-root}
\int_0^T e_i^T(s) R_T e_j^T(s)\, \d s =
\int_0^T K_T e_i^T(s) K_T e_j^T(s)\, \d s
\end{equation}
for all $e_i^T$ and $e_j^T$. Now, condition \eqref{eq:trace-condition} means that $R_T$ is trace class and, consequently, $K_T$ is Hilbert--Schmidt.  In particular, $K_T$ is a compact operator.  Therefore, it admits a kernel. Indeed, a kernel $K_T$ can be defined by using the Mercer expansion \eqref{eq:Mercer} as
\begin{equation}\label{eq:mercer-square-root}
K_T(t,s) = \sum_{i=1}^\infty \sqrt{\lambda_i^T}e_i^T(t)e_i^T(s).
\end{equation}
This kernel is obviously symmetric. Now, it follows that
$$
R(t,s) = \int_0^T K_T(t,u)K_T(s,u)\, \d u,
$$
and the representation \eqref{eq:fredholm} follows from this.  

Finally, let us note that the uniqueness upto a unitary transformation is obvious from the square-root relation \eqref{eq:square-root}.
\end{proof}
\begin{rem}
The Fredholm representation \eqref{eq:fredholm} holds also for infinite intervals, i.e. $T=\infty$, if the trace condition \eqref{eq:trace-condition} holds. Unfortunately, this is seldom the case.
\end{rem}
\begin{rem}
The above proof shows that the Fredholm representation \eqref{eq:fredholm} holds in law.  However, one can also construct the process $X$ via \eqref{eq:fredholm} for a given Brownian motion $W$. In this case, the representation \eqref{eq:fredholm} holds of course in $L^2$.  Finally, note that in general it is not possible to construct the Brownian motion in the representation \eqref{eq:fredholm} from the process $X$.  Indeed, there might not be enough randomness in $X$.  To construct $W$ from $X$ one needs that the indicators $\1_t$, $t\in[0,T]$, belong to the range of the operator $K_T$.  
\end{rem}

\begin{rem}
We remark that the separability of $X$ ensures representation of form \eqref{eq:fredholm} where the kernel $K_T$ only satisfies a weaker condition $K_T(t,\cdot)\in L^2([0,T])$ for all $t\in[0,T]$, which may happen if the trace condition \eqref{eq:trace-condition} fails.  In this case, however, the associated operator $K_T$ does not belong to $L^2([0,T])$, which may be undesirable.
\end{rem}

\begin{exa}\label{exa:degenerate}
Let us consider the following very degenerate case:
Suppose $X_t = f(t)\xi$, where $f$ is deterministic and $\xi$ is a standard normal random variable.  Suppose $T>1$. Then 
\begin{equation}
\label{eq:example_toy}
X_t = \int_0^T f(t)\1_{[0,1)}(s)\, \d W_s.
\end{equation}
So, $K_T(t,s) =f(t)\1_{[0,1)}(s)$.  Now, if $f\in L^2([0,T])$, then condition \eqref{eq:trace-condition} is satisfied and $K_T\in L^2([0,T]^2)$.  On the other hand, even if $f\notin L^2([0,T])$ we can still write $X$ in form (\ref{eq:example_toy}). However, in this case the kernel $K_T$ does not belong to $L^2([0,T]^2)$. 
\end{exa}

\begin{exa}
Consider a truncated series expansion 
$$
X_t = \sum_{k=1}^n e_k^T(t)\xi_k, 
$$
where $\xi_k$ are independent standard normal random variables and 
$$
e_k^T(t) = \int_0^t \tilde e_k^T(s)\, \d s,
$$ 
where $\tilde e_k^T$, $k\in\N$, is an orthonormal basis in $L^2([0,T])$. Now it is straightforward to check that this process is not \emph{purely non-deterministic} (see \cite{cramer} for definition) and consequently, $X$ cannot have Volterra representation while it is clear that $X$ admits a Fredholm representation. On the other hand, by choosing the functions $\tilde e_k^T$ to be the trigonometric basis on $L^2([0,T])$, $X$ is a finite-rank approximation of the Karhunen--Lo\`eve representation of standard Brownian motion on $[0,T]$. Hence by letting $n$ tend to infinity we obtain the standard Brownian motion, and hence a Volterra process. 
\end{exa}

\begin{exa}
Let $W$ be a standard Brownian motion on $[0,T]$ and consider the Brownian bridge. Now, there are two representations of the Brownian bridge (see \cite{sy} and references therein on the representations of Gaussian bridges).  The orthogonal representation is 
$$
B_t = W_t-\frac{t}{T}W_T.
$$  
Consequently, $B$ has a Fredholm representation with kernel $K_T(t,s) = \1_t(s)-t/T$.  The canonical representation of the Brownian bridge is
$$
B_t = (T-t)\int_0^t \frac{1}{T-s}\, \d W_s.
$$
Consequently, the Brownian bridge has also a Volterra-type representation with kernel 
$K(t,s) = (T-t)/(T-s)$. 
\end{exa}

%%%%%%%%%%%%%%%%%%%%%%%%%%%%%%%%%%%%%%%%%%%%%%%%%%%%%%%%%%%%%%%%%%%%%%%%%%%%%%%
%%%%%%%%%%%%%%%%%%%%%%%%%%%%%%%%%%%%%%%%%%%%%%%%%%%%%%%%%%%%%%%%%%%%%%%%%%%%%%%
\section{Transfer Principle and Stochastic Analysis}\label{sect:sa}

%%%%%%%%%%%%%%%%%%%%%%%%%%%%%%%%%%%%%%%%%%%%%%%%%%%%%%%%%%%%%%%%%%%%%%%%%%%%%%%
\subsection{Wiener Integrals}\label{subs:wi}

The following Theorem \ref{thm:transfer-principle-wi} is the transfer principle in the context of Wiener integrals.  The same principle extends to multiple Wiener integrals and Malliavin calculus later in the following subsections.

Recall that for any kernel $\Gamma\in L^2([0,T]^2)$ its associated operator on $L^2([0,T])$ is
$$
\Gamma f(t) = \int_0^T f(s)\Gamma(t,s)\, \d s.
$$

\begin{dfn}[Adjoint associated operator]\label{dfn:adjoint-associated}
The adjoint associated operator $\Gamma^*$ of a kernel $\Gamma\in L^2([0,T]^2)$ is defined by linearly extending the relation
$$
\Gamma^*\1_t = \Gamma(t,\cdot).
$$
\end{dfn}

\begin{rem}
\label{rem:adjoint}
The name and notation of ``adjoint'' for $K_T^*$ comes from Al\`os, Mazet and Nualart \cite{alos-mazet-nualart} where they showed that in their Volterra context $K_T^*$ admits a kernel and is an adjoint of $K_T$ in the sense that
$$
\int_0^T K_T^*f(t)\, g(t)\, \d t =
\int_0^T f(t)\, K_Tg(\d t). 
$$
for step-functions $f$ and $g$ belonging to $L^2([0,T])$. It is straightforward to check that this statement is valid also in our case.
\end{rem}

\begin{exa}
Suppose the kernel $\Gamma(\cdot,s)$ is of bounded variation for all $s$ and that $f$ is nice enough. Then
$$
\Gamma^* f(s) = \int_0^T f(t)\Gamma(\d t,s).
$$ 
\end{exa}

\begin{thm}[Transfer principle for Wiener integrals]\label{thm:transfer-principle-wi}
Let $X$ be a separable centered Gaussian process with representation \eqref{eq:fredholm} and let $f\in \H_T$. Then
$$
\int_0^T f(t)\, \d X_t = \int_0^T K^*_T f (t)\, \d W_t.
$$
\end{thm}

\begin{proof}
Assume first that $f$ is an elementary function of form
$$
f(t) = \sum_{k=1}^n a_k \textbf{1}_{A_k}
$$
for some disjoint intervals $A_k=(t_{k-1},t_k]$. Then the claim follows by the very definition of the operator $K_T^*$ and Wiener integral with respect to $X$ together with representation \eqref{eq:fredholm}. Furthermore, this shows that $K_T^*$ provides an isometry between $\H_T$ and $L^2([0,T])$. Hence $\H_T$ can be viewed as a closure of elementary functions with respect to $\| f\|_{\H_T} = \| K_T^*f\|_{L^2([0,T])}$ which proves the claim. 
\end{proof}

%%%%%%%%%%%%%%%%%%%%%%%%%%%%%%%%%%%%%%%%%%%%%%%%%%%%%%%%%%%%%%%%%%%%%%%%%%%%%%%
\subsection{Multiple Wiener Integrals}\label{subs:mwi}

The study of multiple Wiener integrals go back to It\^o \cite{ito} who studied the case of Brownian motion.  Later Huang and Cambanis \cite{huang-cambanis} extended to notion to general Gaussian processes.  Dasgupta and Kallianpur \cite{dasgupta-kallianpur1,dasgupta-kallianpur2} and Perez-Abreu and Tudor \cite{perezabreu-tudor} studied multiple Wiener integrals in the context of fractional Brownian motion.  In \cite{dasgupta-kallianpur1,dasgupta-kallianpur2} a method that involved a prior control measure was used and in \cite{perezabreu-tudor} a transfer principle was used. Our approach here extends the transfer principle method used in \cite{perezabreu-tudor}.

We begin by recalling multiple Wiener integrals with respect to Brownian motion and then we apply transfer principle to generalize the theory to arbitrary Gaussian process.

%%%%%%%%%%%%%%%%%%%%%%%%%%%%%%%%%%%%%%%  COPY-PASTETTUA HÖLÖTYSTÄ VANHASTA JUTUSTA

Let $f$ be a elementary function on $[0,T]^p$ that vanishes on the diagonals, i.e. 
$$
f = \sum_{i_1,\ldots,i_p=1}^{n}
a_{i_1\ldots i_p} \1_{\Delta_{i_1}\times\cdots\times\Delta_{i_p}},
$$
where $\Delta_{k} := [t_{k-1},t_k)$ and $a_{i_1\ldots i_p} = 0$ whenever $i_k = i_\ell$ for some $k\ne \ell.$ For such $f$ we define the multiple Wiener integral as
\begin{eqnarray*}
I^{W}_{T,p}(f) &:=&
\int_0^T\cdots\int_0^T f(t_1,\ldots,t_p) \,
\delta W_{t_1}\cdots\delta W_{t_p} \\
&:=&
\sum_{i_1,\ldots,i_p=1}^{n}
a_{i_1\ldots i_p} \Delta W_{t_1}\cdots\Delta W_{t_p}, 
\end{eqnarray*}
where we have denoted $\Delta W_{t_k} := W_{t_k} - W_{t_{k-1}}.$
For $p=0$ we set $I^{W}_0(f) = f.$ Now, it can be shown that elementary functions that vanish on the diagonals are dense in $L^2([0,T]^p).$ Thus, one can extend the operator $I^{W}_{T,p}$ to the space
$L^2([0,T]^p)$. This extension is called the \emph{multiple Wiener integral} with respect to the Brownian motion.
\begin{rem}
It is well known that $I^{W}_{T,p}(f)$ can be understood as a
multiple of iterated Ito integral if and only if $f(t_1,\ldots,t_p) = 0$
unless $t_1\le\cdots\le t_p.$ In this case we have
$$
I^{W}_{T,p}(f) = p!
\int_0^T\int_0^{t_p}\cdots\int_0^{t_2}
f(t_1,\ldots,t_p)\,\d W_{t_1}\cdots\d W_{t_p}.
$$  
For the case of Gaussian processes that are not martingales this fact is totally useless.
\end{rem}
For a general Gaussian process $X$, recall first the Hermite polynomials:
$$
H_p(x) := \frac{(-1)^p}{p!}\e^{\frac{1}{2}x^2}
\frac{\d^p}{\d x^p}\left(\e^{-\frac{1}{2}x^2}\right).
$$
For any $p \ge 1$ let the $p$th Wiener chaos of $X$ be
the closed linear subspace of $L^2 (\Omega)$ generated by the random variables 
$\{ H_p \left( X(\varphi) \right),\ \varphi \in \mathcal{H}, \ \Vert \varphi \Vert_{\mathcal{H}} = 1\}$, where $H_p$ is the $p$th Hermite polynomial. It is well known that
the mapping $I_{p}^{X}(\varphi^{\otimes n}) = n! H_p \left( X(\varphi)\right)$ provides a linear isometry between the symmetric tensor product $\H^{\odot p}$
and the $p$th Wiener chaos. The random variables $I_{p}^{X}(\varphi^{\otimes p})$ are called \textit{multiple Wiener} integrals of order $p$ with respect to the Gaussian process $X$.

%%%%%%%%%%%%%%%%%%%%%%%%%%%%%%%%%%%%%%%

Let us now consider the multiple Wiener integrals $I_{T,p}$ for a general Gaussian process $X$. We \emph{define} the multiple integral $I_{T,p}$ by using the transfer principle in Definition \ref{dfn:transfer-principle-mwi} below and later argue that this is the ``correct'' way of defining them.
So, let $X$ be a centered Gaussian process on $[0,T]$ with covariance $R$ and representation \eqref{eq:fredholm} with kernel $K_T$.

\begin{dfn}[$p$-fold adjoint associated operator]\label{dfn:p-adjoint-associated}
Let $K_T$ be the kernel in \eqref{eq:fredholm} and let $K_T^*$ be its adjoint associated operator. Define
\begin{equation}\label{eq:k-star-tensor}
K_{T,p}^* := \left(K_T^*\right)^{\otimes p}.
\end{equation}
In the same way, define
$$
\H_{T,p} := \H_T^{\otimes p} \quad\mbox{ and }\quad
\H_{T,p}^0 := (\H_T^{\otimes p})^0. 
$$
Here the tensor products are understood in the sense of Hilbert spaces, i.e., they are closed under the inner product corresponding to the $p$-fold product of the underlying inner-product.
\end{dfn}

\begin{dfn}\label{dfn:transfer-principle-mwi}
Let $X$ be a centered Gaussian process with representation \eqref{eq:fredholm} and let $f\in \H_{T,p}$. Then
$$
I_{T,p}(f) := I_{T,p}^W\left(K_{T,p}^*f\right) 
$$
\end{dfn}

The following example should convince the reader that this is indeed the correct definition.

\begin{exa}
Let $p=2$ and let $h=h_1\otimes h_2$, where both $h_1$ and $h_2$ are step-functions. Then $$
(K_{T,2}^*h)(x,y) = (K_T^*h_1)(x)(K_T^*h_2)(y)
$$ 
and 
\begin{equation*}
\begin{split}
&I_{T,2}^W\left(K_{T,2}^*f\right) \\
&= \int_0^T K_T^* h_1(v) \d W_v \cdot \int_0^T K_T^* h_2(u) \d W_u - \la K^*_T h_1,K_T^* h_2\ra_{L^2([0,T])} \\
&= X(h_1)X(h_2)- \la h_1,h_2\ra_{\H_T}
\end{split}
\end{equation*}
as supposed to by analog of the Gaussian martingale case.
\end{exa}

The following proposition shows that our approach to define multiple Wiener integrals is consistent with the traditional approach where multiple Wiener integrals for more general Gaussian process $X$ are defined as the closed linear space generated by Hermite polynomials.

\begin{pro}\label{pro:multiple-hermite}
Let $H_p$ be the $p^{th}$ Hermite polynomial and let $h\in\H_{T}$. Then
$$
I_{T,p}\left(h^{\otimes p}\right) = p!\| h\|_{\H_T}^p H_p\left(\frac{X(h)}{\| h\|_{\H_T}}\right)
$$
\end{pro}

\begin{proof}
First note that without loss of generality we can assume $\| h\|_{\H_T}=1$. Now by the definition of multiple Wiener integral with respect to $X$ we have
$$
I_{T,p}\left(h^{\otimes p}\right) = I^W_{T,p}\left(K^*_{T,p}h^{\otimes p}\right),
$$
where 
$$
K^*_{T,p}h^{\otimes p} = \left(K^*_{T}h\right)^{\otimes p}.
$$
Consequently, by \cite[Proposition 1.1.4]{nualart} we obtain
$$
I_{T,p}\left(h^{\otimes p}\right) = p! H_p\left(W(K^*_T h)\right)
$$
which implies the result together with Theorem \ref{thm:transfer-principle-wi}.
\end{proof}

Proposition \ref{pro:multiple-hermite} extends to the following product formula, which is also well-known in the Gaussian martingale case, but apparently new for general Gaussian processes. Again, the proof is straightforward application of transfer principle.

\begin{pro}\label{pro:multiple-product}
Let $f\in \H_{T,p}$ and $g \in \H_{T,q}$.  Then
\begin{equation}\label{eq:multiple-product}
I_{T,p}(f) I_{T,q}(g) =
\sum_{r=0}^{p\wedge q} r!\binom{p}{r}\binom{q}{r} I_{T,p+q-2r}(f\tilde{\otimes}_{K_T,r} g),
\end{equation}
where
\begin{equation}
\label{eq:tensor_def}
f\tilde{\otimes}_{K_T,r} g = \left(K^*_{T,p+q-2r}\right)^{-1}\left(K^*_{T,p}f \tilde{\otimes}_r K^*_{T,q}g\right)
\end{equation}
and $\left(K^*_{T,p+q-2r}\right)^{-1}$ denotes the pre-image of $K^*_{T,p+q-2r}$.
\end{pro}

\begin{proof}
The proof follows directly from the definition of $I_{T,p}(f)$ and \cite[Proposition 1.1.3]{nualart}.
\end{proof}

\begin{exa}
Let $f\in\H_{T,p}$ and $g\in\H_{T,q}$ be of forms $f(x_1,\ldots,x_p) = \prod_{k=1}^p f_k(x_k)$ and $g(y_1,\ldots,y_p) = \prod_{k=1}^q g_k(y_k)$. Then
\begin{equation*}
\begin{split}
&\left(K^*_{T,p}f \tilde{\otimes}_r K^*_{T,q}g\right)(x_1,\ldots, x_{p-r},y_1,\ldots,y_{q-r})\\
&= \int_{[0,T]^r}\prod_{k=1}^{p-r}K^*_Tf_k(x_k)\prod_{j=1}^r K^*_T f_{p-r+j}(s_j) 
\\
&\phantom{=}\times
\prod_{k=1}^{q-r}K^*_Tg_k(y_k)\prod_{j=1}^{r} K^*_T g_{q-r+j}(s_j)\,\d s_1\cdots \d s_r\\
&= \prod_{k=1}^{p-r}K^*_Tf_k(x_k)\prod_{k=1}^{q-r}K^*_Tg_k(y_k)
\la \prod_{j=1}^r f_{p-r+j},\prod_{j=1}^r g_{q-r+j}\ra_{\H_{T,r}}.
\end{split}
\end{equation*}
Hence 
$$
f\tilde{\otimes}_{K_T,r} g = \prod_{k=1}^{p-r}f_k(x_k)\prod_{k=1}^{q-r}g_k(y_k)\la \prod_{j=1}^r f_j,\prod_{j=1}^r g_j\ra_{\H_{T,r}}
$$
as supposed to.
\end{exa}

\begin{rem}
In the literature multiple Wiener integrals are usually defined as the closed linear space spanned by Hermite polynomials. In such a case Proposition \ref{pro:multiple-hermite} is clearly true by the very definition. Furthermore, one has a multiplication formula (see e.g. \cite{pec-nourd})
$$
I_{T,p}(f) I_{T,q}(g) =
\sum_{r=0}^{p\wedge q} r!\binom{p}{r}\binom{q}{r} I_{T,p+q-2r}(f\tilde{\otimes}_{\H_T,r} g),
$$
where $f\tilde{\otimes}_{\H_T,r} g$ denotes symmetrization of tensor product
$$
f\otimes_{\H_T,r} g = \sum_{i_1,\ldots,i_r=1}^\infty \la f,e_{i_1} \otimes \ldots \otimes e_{i_r}\ra_{\H_T^{\otimes r}} \otimes \la g,e_{i_1} \otimes \ldots \otimes e_{i_r}\ra_{\H_T^{\otimes r}}.
$$
and $\{e_k,k=1,\ldots\}$ is a complete orthonormal basis of the Hilbert space $\H_T$. Clearly, by Proposition \ref{pro:multiple-hermite}, both formulas coincide. This also shows that \eqref{eq:tensor_def} is well-defined.
\end{rem}

%%%%%%%%%%%%%%%%%%%%%%%%%%%%%%%%%%%%%%%%%%%%%%%%%%%%%%%%%%%%%%%%%%%%%%%%%%%%%%%
\subsection{Malliavin calculus and Skorohod integrals}\label{subs:mc}
We begin by recalling some basic facts on Malliavin calculus.

\begin{dfn}
Denote by $\mathcal{S}$ the space of all smooth random variables of the form
\begin{equation*}
F= f(X(\varphi_1), \cdots, X(\varphi_n)), \qquad \varphi_1, \cdots, \varphi_n \in \H_T,
\end{equation*}
where $f \in C_{b}^{\infty}(\R^n)$ i.e. $f$ and all its derivatives are bounded. The Malliavin derivative $D_T F=D_T^{X}F$ of $F$ is an element of 
$L^{2}(\Omega;\H_T)$ defined by
\begin{equation*}
D_T F= \sum_{i=1}^{n} \partial_{i} f (X(\varphi_1), \cdots, X(\varphi_n)) \varphi_i.
\end{equation*}
In particular, $D_T X_t = \1_t$.
\end{dfn}

\begin{dfn}
Let $\mathbb{D}^{1,2}= \mathbb{D}_X^{1,2}$ be the Hilbert space of all square integrable Malliavin differentiable random variables defined as the 
closure of $\mathcal{S}$ with respect to norm 
\begin{equation*}
{\Vert F \Vert}_{1,2}^{2} = \E\left[|F|^{2}\right] + \E\left[\Vert D_T F \Vert_{\H_T}^{2}\right].
\end{equation*}
\end{dfn}

The divergence operator $\delta_T$ is defined as the adjoint operator of the Malliavin derivative $D_T$.

\begin{dfn}
The domain $\text{Dom } \delta_T$ of the operator $\delta_T$ is the set of random variables $u\in L^{2}(\Omega;\H_T)$ satisfying
\begin{equation}
\label{eq:domain_def}
\big| \E \langle D_T F, u \rangle_{\H_T} \big| \le c_{u} \Vert F \Vert_{L^{2}}
\end{equation}
for any $F \in \mathbb{D}^{1,2}$ and some constant $c_u$ depending only on $u$. For $u\in\text{Dom } \delta_T$ the divergence operator $\delta_T(u)$ is a square integrable random variable defined by the duality relation
\begin{equation*}
 \E\left[F \delta_T(u)\right] = \E {\langle D_T F, u \rangle}_{\H_T}
\end{equation*}
for all $F \in \mathbb{D}^{1,2}$.
\end{dfn}
\begin{rem}
It is well-known that $\mathbb{D}^{1,2}\subset \text{Dom } \delta_T$.
\end{rem}

We use the notation
\begin{equation*}
\delta_T(u) = \int_0^T u_s\,\delta X_s.
\end{equation*}

\begin{thm}[Transfer principle for Malliavin calculus]\label{thm:transfer-principle-mc}
Let $X$ be a separable centered Gaussian process with Fredholm representation \eqref{eq:fredholm}.  Let $D_T$ and $\delta_T$ be the Malliavin derivative and the Skorohod integral with respect to $X$ on $[0,T]$.  Similarly, let $D_T^W$ and $\delta_T^W$ be the Malliavin derivative and the Skorohod integral with respect to the Brownian motion $W$ of \eqref{eq:fredholm} restricted on $[0,T]$. Then
$$
\delta_T = \delta_T^W K_T^* \quad\mbox{ and }\quad K_T^* D_T = D_T^W.
$$
\end{thm}

\begin{proof}
The proof follows directly from transfer principle and the isometry provided by $K^*_T$ with same arguments as in \cite{alos-mazet-nualart}. Indeed, by isometry we have
$$
\H_T = (K^*_T)^{-1}(L^2([0,T])),
$$
where $(K^*_T)^{-1}$ denotes the pre-image, which implies that
$$
\mathbb{D}^{1,2}(\H_T) = (K^*_T)^{-1}(\mathbb{D}^{1,2}_W(L^2([0,T])))
$$
which justifies $K^*_T D_T = D_T^W$.
Furthermore, we have relation
$$
\E{\la u,D_T F\ra}_{\H_T} = \E{\la K_T^*u,D^W_T F\ra}_{L^2([0,T])}
$$
for any smooth random variable $F$ and $u\in L^2(\Omega;\H_T)$.
Hence, by the very definition of $\text{Dom }\delta$ and transfer principle, we obtain
$$
\text{Dom }\delta_T = (K_T^*)^{-1}(\text{Dom }\delta_T^W)
$$
and $\delta_T(u) = \delta_T^W(K_T^*u)$ proving the claim. 
\end{proof}

Now we are ready show that the definition of the multiple Wiener integral $I_{T,p}$ in Subsection \ref{subs:mwi} is correct in the sense that it agrees with the iterated Skorohod integral.

\begin{pro}\label{pro:mwi-skorohod}
Let $h\in \H_{T,p}$ be of form $h(x_1,\ldots,x_p)=\prod_{k=1}^p h_k(x_k)$. Then $h$ is iteratively $p$ times Skorohod integrable and
\begin{equation}
\label{eq:iterative_skorohod}
\int_0^T\!\cdots\!\int_0^T h(t_1,\ldots,t_p)\, 
\delta X_{t_1}\cdots\delta X_{t_p}
= I_{T,p}(h)
\end{equation}
Moreover, if $h\in \H_{T,p}^0$ is such that it is $p$ times iteratively Skorohod integrable, then \eqref{eq:iterative_skorohod} still holds. 
\end{pro}

\begin{proof}
Again the idea is to use the transfer principle together with induction. Note first that the statement is true for $p=1$ by definition and assume next that the statement is valid for $k=1,\ldots,p$. We denote $f_j=\prod_{k=1}^j h_k(x_k)$. Hence, by induction assumption, we have
$$
\int_0^T\!\cdots\!\int_0^T h(t_1,\ldots,t_p,t_{v})\, 
\delta X_{t_1}\cdots\delta X_{t_p}\delta X_v = \int_0^T I_{T,p}(f_{p})h_{p+1}(v)\,\delta X_v.
$$
Put now $F=I_{T,p}(f_{p})$ and $u(t)=h_{p+1}(t)$. Hence by \cite[Proposition 1.3.3]{nualart} and by applying the transfer principle we obtain that $Fu$ belongs to $\mathrm{Dom}\,\delta_T$ and
\begin{equation*}
\begin{split}
\delta_T(Fu) &= \delta_T(u)F - {\la D_tF,u(t)\ra}_{\H_T}\\
&=I^W_{T}(K_T^*h_{p+1})I^W_{T,p}(K^*_{T,p}f_p) - p{\la I_{T,p-1}(f_{p}(\cdot,t)),h_{p+1}(t)\ra}_{\H_T}\\
&= I_{T}(h_{p+1})I_{T,p}(f_p) - pI^W_{T,p-1}(K_{T,p-1}^*f_p\tilde\otimes_1 K_{T}^* h_{p+1})\\
&=I_{T}(h_{p+1})I_{T,p}(f_p) - pI_{T,p-1}(f_p\tilde\otimes_{K_T,1} h_{p+1}).
\end{split}
\end{equation*}
Hence the result is valid also for $p+1$ by Proposition \ref{pro:multiple-product} with $q=1$. 

The claim for general $h\in \H_{T,p}^0$ follows by approximating with a products of simple function. 
\end{proof}

\begin{rem}
Note that \eqref{eq:iterative_skorohod} does not hold for arbitrary $h\in\H_{T,p}^0$ in general without the a priori assumption of $p$ times iterative Skorohod integrability. For example, let $p=2$, $X=B^H$ be a fractional Brownian motion with $H\leq \frac14$ and define $h_t(s,v)= \1_t(s)\1_s(v)$ for some fixed $t\in [0,T]$. Then
$$
\int_0^T h_t(s,v)\, \delta X_v = X_s\1_t(s) 
$$
But $X_\cdot\1_t$ does not belong to $\mathrm{Dom}\,\delta_T$ (see \cite{che-nua}).
\end{rem}

We end this section by providing an extension of It\^o formulas provided by Al\`os, Mazet and Nualart \cite{alos-mazet-nualart}. They considered Gaussian Volterra processes, i.e., they assumed the representation 
$$
X_t = \int_0^t K(t,s)\, \d W_s,
$$
where the Kernel $K$ satisfied certain technical assumptions. In \cite{alos-mazet-nualart}
it was proved that in the case of Volterra processes one has
\begin{equation}
\label{eq:ito}
f(X_t) = f(0) + \int_0^t f'(X_s)\,\delta X_s + 
\frac{1}{2} \int_0^t f''(X_s)\, \d R(s,s)
\end{equation}
if $f$ satisfies the growth condition
\begin{equation}
\label{eq:growth}
\max\left[|f(x)|,|f'(x)|,|f''(x)|\right] \le ce^{\lambda |x|^2}
\end{equation}
for some $c>0$ and $\lambda < \frac{1}{4}\left(\sup_{0\leq s \leq T}\E X_s^2\right)^{-1}$. In the following we will consider different approach which ables us to:
\begin{enumerate}
\item
prove that such formula holds with minimal requirements,
\item
give more instructive proof of such result,
\item
extend the result from Volterra context to more general Gaussian processes,
\item
drop some technical assumptions posed in \cite{alos-mazet-nualart}.
\end{enumerate}

For simplicity, we assume that the variance of $X$ is of bounded variation to guarantee the existence of the integral 
\begin{equation}\label{eq:R-integral}
\int_0^T f''(X_t)\,\d R(t,t).
\end{equation}
If the variance is not of bounded variation, then the integral \eqref{eq:R-integral} may be understood by integration by parts if $f''$ is smooth enough, or in the general case, via the inner product $\la\cdot,\cdot\ra_{\H_T}$. In Theorem \ref{thm:ito-skorohod-formula} we also have to assume that the variance of $X$ is bounded.

The result for polynomials is straightforward, once we assume that the paths of polynomials of $X$ belong to $L^2(\Omega;\H_T)$.

\begin{pro}[It\^o formula for polynomials]\label{pro:ito-skorohod-formula-polynomials}
Let $X$ be a separable centered Gaussian process with covariance $R$ and assume that $p$ is a polynomial. Furthermore, assume that for each polynomial $p$ we have $p(X_\cdot)\1_t\in L^2(\Omega;\H_T)$. Then for each $t\in[0,T]$ we have
\begin{equation}
\label{eq:ito-polynomial}
p(X_t) = p(X_0) + \int_0^t p'(X_s)\,\delta X_s + 
\frac{1}{2} \int_0^t p''(X_s)\, \d R(s,s)
\end{equation}
if and only if $X_\cdot\1_t$ belongs to $\mathrm{Dom}\,\delta_T$.
\end{pro}

\begin{rem}
The message of the above result is that once the processes $p(X_\cdot)\1_t \in L^2(\Omega;\H_T)$, then they automatically belong to the domain of $\delta_T$ which is a subspace of $L^2(\Omega;\H_T)$. However, in order to check $p(X_\cdot)\1_t \in L^2(\Omega;\H_T)$ one needs more information on the Kernel $K_T$. A sufficient condition is provided in Corollary \ref{cor:amn-fredholm_extension} which covers many cases of interest.
\end{rem}
\begin{proof}
By definition and applying transfer principle, we have to prove that $p'(X_\cdot)\1_t$ belongs to domain of $\delta_T$ and that 
\begin{eqnarray}
\label{eq:delta_def-polynomial}
\lefteqn{\E \int_0^t D_s G K_T^*[p'(X_\cdot)\1_t]\,\d s}\\
&=&
\E\left[Gp(X_t)\right]-\E\left[Gp(X_0)\right]
-\frac{1}{2}\int_0^t \E\left[Gp''(X_t)\right]\,\d R(t,t) \nonumber  
\end{eqnarray}
for every random variable $G$ from a total subset of $L^2(\Omega)$. In other words, it is sufficient to show that \eqref{eq:delta_def} is valid for random variables of form $G=I_n^W(h^{\otimes n})$, where $h$ is a step function. 

Note first that it is sufficient to prove the claim only for Hermite polynomials $H_k,k=1,\ldots$. Indeed, it is well-known that any polynomial can be expressed as a linear combination of Hermite polynomials and consequently, the result for arbitrary polynomial $p$ follows by linearity.

We proceed by induction. First it is clear that first two polynomials $H_0$ and $H_1$ satisfies \eqref{eq:delta_def-polynomial}. Furthermore, by assumption $H'_2(X_\cdot)\1_t$ belongs to $\mathrm{Dom}\,\delta_T$ from which \eqref{eq:delta_def-polynomial} is easily deduced by \cite[Proposition 1.3.3]{nualart}. Assume next that the result is valid for Hermite polynomials $H_k,k=0,1,\ldots n$. Then, recall well-known recursion formulas
\begin{eqnarray*}
H_{n+1}(x) &=& xH_{n}(x) - n H_{n-1}(x), \\
H'_n(x) &=& nH_{n-1}(x). 
\end{eqnarray*}
The induction step follows with straightforward calculations by using the recursion formulas above and \cite[Proposition 1.3.3]{nualart}. We leave the details to the reader.
\end{proof}

We will now illustrate how the result can be generalized for functions satisfying the growth condition \eqref{eq:growth} by using Proposition \ref{pro:ito-skorohod-formula-polynomials}. 
First note that the growth condition \eqref{eq:growth} is indeed natural since it guarantees that the left side of \eqref{eq:ito} is square integrable. Consequently, since operator $\delta_T$ is a mapping from $L^2(\Omega;\H_T)$ into $L^2(\Omega)$, functions satisfying \eqref{eq:growth} are largest class of functions for which  \eqref{eq:ito} can hold. However, it is not clear in general whether $f'(X_\cdot)\1_t$ belongs to $\mathrm{Dom}\,\delta_T$. Indeed, for example in \cite{alos-mazet-nualart} the authors posed additional conditions on the Volterra kernel $K$ to guarantee this. As our main result we show that $\E\Vert f'(X_\cdot)\1_t\Vert^2_{\H_T}<\infty$ implies that \eqref{eq:ito} holds. In other words, the It\^o formula \eqref{eq:ito} is not only natural but it is also the only possibility. 

\begin{thm}[It\^o formula for Skorohod integrals]\label{thm:ito-skorohod-formula}
Let $X$ be a separable centered Gaussian process with covariance $R$ such that all the polynomials $p(X_\cdot)\1_t \in L^2(\Omega;\H_T)$
Assume that $f\in C^2$ satisfies growth condition \eqref{eq:growth} and that the variance of $X$ is bounded and of bounded variation. If 
\begin{equation}
\label{eq:hilbert_condition}
\E \Vert f'(X_\cdot)\1_t\Vert_{\H_T}^2 < \infty
\end{equation}
for any $t\in[0,T]$, then 
$$
f(X_t) = f(X_0) + \int_0^t f'(X_s)\,\delta X_s + 
\frac{1}{2} \int_0^t f''(X_s)\, \d R(s,s).
$$
\end{thm}

\begin{proof}
In this proof we assume, for notational simplicity and with no loss of generality, that $\sup_{0\leq s\leq T}R(s,s)=1$.

First it is clear that \eqref{eq:hilbert_condition} implies that $f'(X_\cdot)\1_t$ belongs to domain of $\delta_T$. Hence we only have to prove that
\begin{eqnarray*}
\label{eq:delta_def}
\lefteqn{\E {\la D_T G, f'(X_\cdot)\1_t\ra}_{\H_T}} \\
&=& \E[Gf(X_t)]-\E[Gf(X_0)] 
-\frac{1}{2}\int_0^t \E[Gf''(X_s)]\,\d R(s,s). \nonumber
\end{eqnarray*}
for every random variable $G=I_n^W(h^{\otimes n})$. 

Now, it is well-known that Hermite polynomials, when properly scaled, form an orthogonal system in $L^2(\R)$ when equipped with the Gaussian measure.  Now each $f$ satisfying the growth condition \eqref{eq:growth} have a series representation 
\begin{equation}
\label{eq:f_series}
f(x) = \sum_{k=0}^\infty \alpha_k H_k(x).
\end{equation}
Indeed, the growth condition \eqref{eq:growth} implies that
$$
\int_\R |f'(x)|^2e^{-\frac{x^2}{2\sup_{0\leq s\leq T}R(s,s)}}\d x < \infty.
$$

Furthermore, we have 
$$
f(X_s) = \sum_{k=0}^\infty \alpha_k H_k(X_s)
$$
where the series converge almost surely and in $L^2(\Omega)$, and similar conclusion is valid for derivatives $f'(X_s)$ and $f''(X_s)$. 

Then, by applying \eqref{eq:hilbert_condition} we obtain that for any $\epsilon>0$ there exists $N=N_\epsilon$ such that we have
$$
\E{\la D_T G, f'_n(X_\cdot)\1_t\ra}_{\H_T} < \epsilon,\quad n\geq N
$$
where 
$$
f'_n(X_s) = \sum_{k=n}^\infty \alpha_k H'_k(X_s).
$$
Consequently, for random variables of form $G=I_n^W(h^{\otimes n})$ we obtain, by choosing $N$ large enough and applying Proposition \ref{pro:ito-skorohod-formula-polynomials}, that
\begin{equation*}
\begin{split}
&\E[Gf(X_t)]-\E[Gf(X_0)] -\frac{1}{2}\int_0^t \E (Gf''(X_t))\,\d R(t,t) \\
&- \E {\la D_T G, f'(X_\cdot)\1_t\ra}_{\H_T}\\
&=\E {\la D_T G, f'_n(X_\cdot)\1_t\ra}_{\H_T}\\
&<\epsilon.
\end{split}
\end{equation*}
Now the left side does not depend on $n$ which concludes the proof.
\end{proof}

\begin{rem}
Note that actually it is sufficient to have
\begin{equation}
\label{eq:ito-iff}
\E{\la D_T G, f'(X_\cdot)\1_t\ra}_{\H_T} = \sum_{k=1}^\infty \alpha_k\E{\la D_T G, H'_k(X_\cdot)\1_t\ra}_{\H_T}
\end{equation}
from which the result follows by Proposition \ref{pro:ito-skorohod-formula-polynomials}. Furthermore, taking account growth condition \eqref{eq:growth} this is actually sufficient and necessary condition for formula \eqref{eq:ito} to hold. Consequently, our method can also be used to obtain It\^o formulas by considering \emph{extended} domain of $\delta_T$ (see \cite{kruk-russo} or \cite{lei-nualart}). This is the topic of Subsection \ref{subs:ed} below.
\end{rem}

\begin{exa}
It is known that if $X=B^H$ is a fractional Brownian motion with $H>\frac{1}{4}$, then $f'(X_\cdot)\1_t$ satisfies condition \eqref{eq:hilbert_condition} while for $H\leq \frac{1}{4}$ it does not (see \cite[Chapter 5]{nualart}). Consequently, a simple application of Theorem \ref{thm:ito-skorohod-formula} covers fractional Brownian motion with $H>\frac14$. For the case $H\leq \frac14$ one has to consider extended domain of $\delta_T$ which is proved in \cite{kruk-russo}. Consequently, in this case we have \eqref{eq:ito-iff} for any $F\in \mathcal{S}$.
\end{exa}
We end this section by illustrating the power of our method with the following simple corollary which is an extension of \cite[Theorem 1]{alos-mazet-nualart}. 
\begin{cor}
\label{cor:amn-fredholm_extension}
Let $X$ be a separable centered continuous Gaussian process with covariance $R$ that is bounded and such that the Fredholm kernel $K_T$ is of bounded variation and
$$
\int_0^T \left(\int_0^T \Vert X_t - X_s\Vert_{L^2(\Omega)} |K_T|(\d t,s)\right)^2\d s < \infty.
$$
Then for any $t\in[0,T]$ we have
$$
f(X_t) = f(X_0) + \int_0^t f'(X_s)\,\delta X_s + 
\frac{1}{2} \int_0^t f''(X_s)\, \d R(s,s).
$$
\end{cor}
\begin{proof}
Note that assumption is a Fredholm version of condition \textbf{(K2)} in \cite{alos-mazet-nualart} which implies condition \eqref{eq:hilbert_condition}. Hence the result follows by Theorem \ref{thm:ito-skorohod-formula}.
\end{proof}

%%%%%%%%%%%%%%%%%%%%%%%%%%%%%%%%%%%%%%%%%%%%%%%%%%%%%%%%%%%%%%%%%%%%%%%%%%%%%%%
\subsection{Extended divergence operator}\label{subs:ed}

As shown in Subsection \ref{subs:mc} the It\^o formula \eqref{eq:ito} is the only possibility. However, the problem is that the space $L^2(\Omega;\H_T)$ may be to small to contain the elements $f'(X_\cdot)\1_t$. In particular, it may happen that not even the process $X$ itself belong to $L^2(\Omega;\H_T)$ (see e.g. \cite{che-nua} for the case of fractional Brownian motion with $H\leq \frac14$). This problem can be overcome by considering an extended domain of $\delta_T$. The idea of extended domain is to extend the inner product $\la u,\varphi\ra_{\H_T}$ for simple $\varphi$ to more general processes $u$ and then define extended domain by \eqref{eq:domain_def} with a restricted class of test variables $F$. This also gives another intuitive reason why extended domain of $\delta_T$ can be useful; indeed, here we have proved that It\^o formula \eqref{eq:ito} is the only possibility, and what one essentially needs for such result is that 
\begin{enumerate}
\item
$X_\cdot\1_t$ belongs to $\mathrm{Dom}\,\delta_T$,

\item 
equation \eqref{eq:ito-iff} is valid for functions satisfying \eqref{eq:growth}.
\end{enumerate}
Consequently, one should look for extensions of operator $\delta_T$ such that these two things are satisfied.

To facility the extension of domain, we make the following relatively moderate assumption:

\bigskip

\noindent
\textbf{(H)}\quad The function $t\mapsto R(t,s)$ is of bounded variation on $[0,T]$ and
$$
\sup_{t\in[0,T]}\int_0^T |R|(\d s,t) < \infty. 
$$

\bigskip

\begin{rem}
Note that we are making the assumption on the covariance $R$, not the Kernel $K_T$. Hence our case is different from that of \cite{alos-mazet-nualart}.  Also, \cite{lei-nualart} assumed absolute continuity in $R$; we are satisfied with bounded variation.
\end{rem}

We will follow the idea from Lei and Nualart \cite{lei-nualart} and extend the inner product $\la\cdot,\cdot\ra_{\H_T}$ beyond $\H_T$. 

Consider a step function $\varphi$. Then, on the one hand, by the isometry property we have
$$
\la \varphi,\1_t\ra_{\H_T} = \int_0^T (K_T^*\varphi)(s) g_t(s)\,\d s,
$$
where $g_t(s) = K(t,s)\in L^2([0,T])$. On the other hand, by using adjoint property (see Remark \ref{rem:adjoint}) we obtain
$$
\int_0^T (K_T^* \varphi)(s) g_t(s)\, \d s = \int_0^T \varphi(s) \left(K_T g_t\right)(\d s),
$$
where, computing formally, we have 
\begin{eqnarray*}
\left(K_Tg_t\right)(\d s) 
&=& \int_0^T g_t(u) K_T(\d s,u)\d u\\
&=& \int_0^T K_T(t,u) K_T(\d s,u)\d u\\
&=& R(t, \d s).
\end{eqnarray*}
Consequently,
$$
\la \varphi,\1_t\ra_{\H_T} = \int_0^T \varphi(s) R(t,\d s).
$$
This gives motivation to the following definition similar to that of \cite[Definition 2.1]{lei-nualart}.

\begin{dfn}
Denote by $\mathcal{T}_T$ the space of measurable functions $g$ satisfying 
$$
\sup_{t\in[0,T]}\int_0^T |g(s)||R|(t,\d s) < \infty
$$
and let $\varphi$ be a step function of form $\varphi = \sum_{k=1}^n b_k \textbf{1}_{t_k}$. Then we extend $\la \cdot,\cdot\ra_{\H_T}$ to $\mathcal{T}_T$ by defining
$$
\la g,\varphi\ra_{\H_T} = \sum_{k=1}^n b_k \int_0^T g(s)R(t_k,\d s).
$$
In particular, this implies that for $g$ and $\varphi$ as above, we have
$$
\la g\1_t,\varphi\ra_{\H_T} = \int_0^t g(s)\,\d\la \1_s,\varphi\ra_{\H_T}.
$$
\end{dfn}

We define extended domain Dom$^E$ $\delta_T$ similarly as in \cite{lei-nualart}.

\begin{dfn}
A process $u\in L^1(\Omega; \mathcal{T}_T)$ belongs to  $\mathrm{Dom}^E\,\delta_T$ if 
$$
|\E\la u, DF\ra_{\H_T}| \leq c_u \Vert F\Vert_2
$$ 
for any smooth random variable $F\in \mathcal{S}$. In this case, $\delta(u)\in L^2(\Omega)$ is defined by duality relationship
$$
\E[F\delta(u)] = \E\la u,DF\ra_{\H_T}.
$$
\end{dfn}

\begin{rem}
Note that in general $\mathrm{Dom}\,\delta_T$ and $\mathrm{Dom}^E\, \delta_T$
are not comparable. See \cite{lei-nualart} for discussion.
\end{rem}

Note now that if a function $f$ satisfies the growth condition \eqref{eq:growth}, then $f'(X_\cdot)\1_t \in L^1(\Omega;\mathcal{T}_T)$ since \eqref{eq:growth} implies
$$
\E \sup_{t\in[0,T]}|f'(X_t)|^p < \infty
$$
for any $p< \frac{1}{2\lambda}\left(\sup_{t\in[0,T]} R(t,t)\right)^{-1}$. Consequently, with this definition we are able the get rid of the problem that processes might not belong to corresponding $\H_T$-spaces. Furthermore, this implies that the series expansion \eqref{eq:f_series} converges in the norm $L^1(\Omega;\mathcal{T}_T)$ defined by
$$
\E \int_0^T |u(s)||R|(t,\d s)
$$
which in turn implies \eqref{eq:ito-iff}. 
Hence it is straightforward to obtain the following by first showing the result for polynomials and then by approximating in a similar manner as done in the previous Subsection \ref{subs:mc}, but using the extended domain instead.

\begin{thm}[It\^o formula for extended Skorohod integrals]\label{thm:ito-skorohod-formula-extended}\label{thm:ito-extended}
Let $X$ be a separable centered Gaussian process with covariance $R$ and assume that $f\in C^2$ satisfies growth condition \eqref{eq:growth}. Furthermore, assume that \textbf{(H)} holds and that the variance of $X$ is bounded and of bounded variation. Then for any $t\in[0,T]$ the process $f'(X_\cdot)\1_t$ belongs to Dom$^E$ $\delta_T$ and 
$$
f(X_t) = f(X_0) + \int_0^t f'(X_s)\,\delta X_s + 
\frac{1}{2} \int_0^t f''(X_s)\, \d R(s,s).
$$
\end{thm}

\begin{rem}
As an application of Theorem \ref{thm:ito-extended} it is straightforward to derive version of It\^o--Tanaka formula under additional conditions which guarantee that for a certain sequence of functions $f_n$ we have the convergence of term $\frac{1}{2}\int_0^t f_n''(X_s)\,\d R(s,s)$ to the local time. For details we refer to \cite{lei-nualart}, where authors derived such formula under their assumptions. 
\end{rem}

Finally, let us note that the extension to functions $f(t,x)$ is straightforward, where $f$ satisfies the following growth condition.
\begin{equation}
\label{eq:growth2}
\max\left[|f(t,x)|,|\partial_t f(t,x)|,|\partial_x f(t,x)|,|\partial_{xx} f(t,x)|\right] \le ce^{\lambda |x|^2}
\end{equation}
for some $c>0$ and $\lambda < \frac{1}{4}\left(\sup_{0\leq s \leq T}\E X_s^2\right)^{-1}$.

\begin{thm}[It\^o formula for extended Skorohod integrals, II]\label{thm:ito-skorohod-formula-extended2}
Let $X$ be a separable centered Gaussian process with covariance $R$ and assume that $f\in C^{1,2}$ satisfies growth condition \eqref{eq:growth2}. Furthermore, assume that \textbf{(H)} holds and that the variance of $X$ is bounded and of bounded variation. 
Then for any $t\in[0,T]$ the process $\partial_x f(\cdot,X_\cdot)\1_t$ belongs to Dom$^E$ $\delta_T$ and 
\begin{eqnarray*}
f(t,X_t) &=& f(0,X_0) + \int_0^t \partial_x f(s,X_s)\,\delta X_s + 
\int_0^t \partial_t f(s,X_s)\d s \\
& &+\frac{1}{2} \int_0^t \partial_{xx} f(s,X_s)\, \d R(s,s).
\end{eqnarray*}
\end{thm}

\begin{proof}
Taking into account that we have no problems concerning processes to belong to the required spaces, the formula follows by approximating with polynomials of form $p(x)q(t)$ and following the proof of theorem \ref{thm:ito-skorohod-formula}.
\end{proof}

%%%%%%%%%%%%%%%%%%%%%%%%%%%%%%%%%%%%%%%%%%%%%%%%%%%%%%%%%%%%%%%%%%%%%%%%%%%%%%%
%%%%%%%%%%%%%%%%%%%%%%%%%%%%%%%%%%%%%%%%%%%%%%%%%%%%%%%%%%%%%%%%%%%%%%%%%%%%%%%
%%%%%%%%%%%%%%%%%%%%%%%%%%%%%%%%%%%%%%%%%%%%%%%%%%%%%%%%%%%%%%%%%%%%%%%%%%%%%%%

\section{Applications}\label{sect:appl} 

We illustrate how some results transfer easily from the Brownian case to the Gaussian Fredholm processes. 

%%%%%%%%%%%%%%%%%%%%%%%%%%%%%%%%%%%%%%%%%%%%%%%%%%%%%%%%%%%%%%%%%%%%%%%%%%%%%%%
\subsection{Equivalence in Law}\label{subs:el}

The transfer principle has already been used in connection with the equivalence of law of Gaussian processes in e.g. \cite{sottinen} in the context of fractional Brownian motions and in \cite{baudoin-nualart} in the context of Gaussian Volterra processes satisfying certain non-degeneracy conditions.  The following proposition uses the Fredholm representation \eqref{eq:fredholm} to give a sufficient condition for the equivalence of general Gaussian processes in terms of their Fredholm kernels.

\begin{pro}\label{pro:equivalence}
Let $X$ and $\tilde X$ be two Gaussian process with Fredholm kernels $K_T$ and $\tilde K_T$, respectively.  If there exists a Volterra kernel $\ell\in L^2([0,T]^2)$ such that
\begin{equation}\label{eq:equivalence}
\tilde K_T(t,s) = K_T(t,s) -
\int_s^T K_T(t,u)\ell(u,s)\, \d u, 
\end{equation}
then $X$ and $\tilde X$ are equivalent in law. 
\end{pro}

\begin{proof}
Recall that by the Hitsuda representation theorem \cite[Theorem 6.3]{hida-hitsuda} a centered Gaussian process $\tilde W$ is equivalent in law to a Brownian motion  on $[0,T]$ if and only if there exists a kernel $\ell\in L^2([0,T]^2)$ and a Brownian motion $W$ such that $\tilde W$ admits the representation
\begin{equation}\label{eq:hitsuda}
\tilde W_t = W_t - \int_0^t\!\int_0^s \ell(s,u)\, \d W_u \d s.
\end{equation}

Let $X$ have the Fredholm representations 
\begin{equation}\label{eq:X-fredholm}
X_t = \int_0^T K_T(t,s)\, \d W_s. 
\end{equation}
Then $\tilde X$ is equivalent to $X$ if it admits, in law, the representation
\begin{equation}\label{eq:tildeX-equivalence}
\tilde X_t \stackrel{d}{=} \int_0^T K_T(t,s)\, \d\tilde W_s,
\end{equation}
where $\tilde W$ is connected to $W$ of \eqref{eq:X-fredholm} by \eqref{eq:hitsuda}. 

In order to show \eqref{eq:tildeX-equivalence}, let 
$$
\tilde X_t = \int_0^T \tilde K_T(t,s)\, \d W_t'
$$
be the Fredholm representation of $\tilde X$. Here $W'$ is some Brownian motion. Then, by using the connection \eqref{eq:equivalence} and the Fubini theorem, we obtain
\begin{eqnarray*}
\tilde X_t &=& \int_0^T \tilde K_T(t,s)\, \d W_s' \\
&\stackrel{d}{=}& \int_0^T \tilde K_T(t,s)\, \d W_s \\
&=&
\int_0^T \left(K_T(t,s) -
\int_s^T K_T(t,u)\ell(u,s)\, \d u\right)\, \d W_s \\
&=&
\int_0^T K_T(t,s)\, \d W_s - \int_0^T\!\int_s^T K_T(t,u)\ell(u,s)\,\d u\,\d W_s \\
&=&
\int_0^T K_T(t,s)\, \d W_s - \int_0^T\!\int_0^s K_T(t,s)\ell(s,u)\, \d W_u\,\d s \\
&=&
\int_0^T K_T(t,s)\, \d W_s - \int_0^T K_T(t,s)\left(\int_0^s \ell(s,u)\, \d W_u\right)\,\d s \\
&=&
\int_0^T K_T(t,s)\left(\d W_s -\int_0^s\ell(s,u)\,\d W_u \,\d s\right)\\
&=&
\int_0^T K_T(t,s)\, \d\tilde W_s.
\end{eqnarray*}
Thus, we have shown the representation \eqref{eq:tildeX-equivalence}, and consequently the equivalence of $\tilde X$ and $X$.
\end{proof}

%%%%%%%%%%%%%%%%%%%%%%%%%%%%%%%%%%%%%%%%%%%%%%%%%%%%%%%%%%%%%%%%%%%%%%%%%%%%%%%
\subsection{Generalized Bridges}\label{subs:bridges}

We consider the conditioning, or bridging, of $X$ on $N$ linear functionals $\G_T=[G_T^i]_{i=1}^N$ of its paths: 
$$%\begin{equation}\label{eq:awints}
\G_T(X) = \int_0^T \g(t)\, \d X_t
= \left[\int_0^T g_i(t)\, \d X_t\right]_{i=1}^N.
$$%\end{equation} 
We assume, without any loss of generality, that the functions $g_i$ are 
linearly independent.  Also, without loss of generality we assume that $X_0=0$, and the conditioning is on the set $\{\int_0^T \g(t)\,\d X_t=\0\}$ instead of the apparently more general conditioning on the set $\{\int_0^T \g(t)\,\d X_t=\y\}$.  Indeed, see \cite{sy} how to obtain the more general conditioning from this one.

The rigorous definition of a bridge is the following.
 
\begin{dfn}\label{dfn:bridges}
The \emph{generalized bridge measure} $\P^{\g}$ is the regular conditional law
$$
\P^{\g} =
\P^{\g}\left[X\in \, \cdot\, \right] = 
\P\left[X\in \ \cdot\ \bigg| \int_0^T \g(t)\, \d X_t=\0\right].
$$
A \emph{representation of the generalized Gaussian bridge} 
is any process $X^{\g}$ satisfying
$$
\P\left[X^{\g}\in\,\cdot\,\right] =
\P^{\g}\left[X\in \, \cdot\, \right]
=
\P\left[ X \in\ \cdot\  \bigg| \int_0^T \g(t)\, \d X_t=\0\right].
$$
\end{dfn}  

We refer to \cite{sy} for more details on generalized Gaussian bridges. 

There are many different representations for bridges. A very general representation is the so-called \emph{orthogonal representation} given by
$$
X^{\g}_t =  X_t - 
{\lla \1_t,\g\rra}^\top{\lla \g\rra}^{-1}
\int_0^T \g(u)\, \d X_u,
$$
where, by the transfer principle,
\begin{eqnarray*}
{\lla \g \rra}_{ij} 
&:=& 
\Cov\left[\int_0^T g_i(t)\, \d X_t \,,\, \int_0^T g_j(t)\, \d X_t\right]\\
&=&
\int_0^T K_T^* g_i(t)\, K_T^* g_j(t)\, \d t. 
\end{eqnarray*}
A more interesting representation is the so-called \emph{canonical representation} where the filtrations of the bridge and the original process coincide.  In \cite{sy} such representations were constructed for the so-called prediction-invertible Gaussian processes.  In this subsection we show how the transfer principle can be used to construct a canonical-type bridge representation for all Gaussian Fredholm processes.  We start with an example that should make it clear how one uses the transfer principle.

\begin{exa}\label{exa:bb-canonical}
We construct a canonical-type representation for $X^1$, the bridge of $X$ conditioned on $X_T=0$. Assume $X_0=0$. Now, by the Fredholm representation of $X$ we can write the conditioning as
\begin{equation}\label{eq:bb-transfer}
X_T=\int_0^T 1 \, \d X_t = \int_0^T K_T(T,t)\, \d W_t = 0.
\end{equation}
Let us then denote by $W^{1^*}$ the canonical representation of the Brownian bridge with the conditioning \eqref{eq:bb-transfer}.  Then, by \cite[Theorem 4.12]{sy},
$$
\d W^{1^*}_s = \d W_s - 
\int_0^s
 \frac{K_T(T,s) K_T(T,u)}{\int_u^T K_T(T,v)^2\, \d v}
\, \d W_u\, \d s.
$$ 
Now, by integrating against the kernel $K_T$, we obtain from this that
$$
X^1_t = X_t -\int_0^T K_T(t,s)\int_0^s
 \frac{K_T(T,s) K_T(T,u)}{\int_u^T K_T(T,v)^2\, \d v}
\, \d W_u\, \d s. 
$$ 
This canonical-type bridge representation seems to be a new one.
\end{exa}

Let us then denote
$$
\la\!\la \g \ra\!\ra_{ij}(t) :=
\int_t^T g_i(s) g_j(s)\, \d s.
$$
Then, in the same was as Example \ref{exa:bb-canonical} the same way, by applying the transfer principle to \cite[Theorem 4.12]{sy}, we obtain the following canonical-type bridge representation for general Gaussian Fredholm processes.

\begin{pro}\label{pro:ct-bridges}
Let $X$ be a Gaussian process with Fredholm kernel $K_T$ such that $X_0=0$. Then the bridge $X^\g$ admits the canonical-type representation
$$
X^\g_t = X_t -
\int_0^T K_T(t,s) \int_0^s 
|\la\!\la\g^*\ra\!\ra|(s)\, (\g^*)^\top(s)\, \la\!\la \g^*\ra\!\ra^{-1}(s)\,
\frac{\g^*(u)}{\la\!\la\g^*\ra\!\ra(u)}
\, \d W_u\, \d s,
$$ 
where $\g^* = K_T^*\g$.
\end{pro}

%%%%%%%%%%%%%%%%%%%%%%%%%%%%%%%%%%%%%%%%%%%%%%%%%%%%%%%%%%%%%%%%%%%%%%%%%%%%%%%
\subsection{Series Expansions}\label{subs:se}               

The Mercer square root \eqref{eq:mercer-square-root} can be used to build the Karhunen--Lo\`eve expansion for the Gaussian process $X$.  But the Mercer form \eqref{eq:mercer-square-root} is seldom known.  However, if one can find some kernel $K_T$ such that the representation \eqref{eq:fredholm} holds, then one can construct a series expansion for $X$ by using the transfer principle of Theorem \ref{thm:transfer-principle-wi} as follows:

\begin{pro}[Series expansion]\label{pro:series-expansion}
Let $X$ be a separable Gaussian process with representation \eqref{eq:fredholm}. Let $(\phi^T_j)_{j=1}^\infty$ be any orthonormal basis on $L^2([0,T])$. Then $X$ admits the series expansion
\begin{equation}\label{eq:series-expansion}
X_t = \sum_{j=1}^\infty \int_0^T \phi^T_j(s)K_T(t,s)\, \d s\cdot \xi_j,
\end{equation}
where the $(\xi_j)_{j=1}^\infty$ is a sequence of independent standard normal random variables. The series \eqref{eq:series-expansion} converges in $L^2(\Omega)$; and also almost surely uniformly if and only if $X$ is continuous.
\end{pro}

The proof below uses reproducing kernel Hilbert space technique.  For more details on this we refer to \cite{gilsing-sottinen} where the series expansion is constructed for fractional Brownian motion by using the transfer principle.

\begin{proof}
The Fredholm representation \eqref{eq:fredholm} implies immediately that the reproducing kernel Hilbert space of $X$ is the image $K_TL^2([0,T])$ and $K_T$ is actually an isometry from $L^2([0,T])$ to the reproducing kernel Hilbert space of $X$. The $L^2$-expansion \eqref{eq:series-expansion} follows from this due to \cite[Theorem 3.7]{adler} and the equivalence of almost sure convergence of \eqref{eq:series-expansion} and continuity of $X$ follows \cite[Theorem 3.8]{adler}.
\end{proof}

%%%%%%%%%%%%%%%%%%%%%%%%%%%%%%%%%%%%%%%%%%%%%%%%%%%%%%%%%%%%%%%%%%%%%%%%%%%%%%%
\subsection{Stochastic differential equations and maximum likelihood estimators}

Let us briefly discuss the following generalized Langevin equation
\begin{equation*}
\d X^\theta_t = -\theta X^\theta_t \d t + \d X_t, \quad t\in[0,T]
\end{equation*}
with some Gaussian noise $X$, parameter $\theta>0$, and initial condition $X_0$. This can be written in the integral form
\begin{equation}
\label{eq:SDE_langevin_integral}
X^\theta_t = X^\theta_0 -\theta \int_0^t X^\theta_s \,\d s+ \int_0^T \1_t(s)\, \d X_s.
\end{equation}
Here the integral $\int_0^T \1_t(s)\,\d X_s$ can be understood in a pathwise sense or in a Skorohod sense, and both integrals coincides. Suppose now that the Gaussian noise $X$ has the Fredholm representation
$$
X_t = \int_0^T K_T(t,s)\,\d W_s.
$$
By applying the transfer principle we can write \eqref{eq:SDE_langevin_integral} as
$$
X^\theta_t = X^\theta_0 - \theta \int_0^t X^\theta_s \,\d t + \int_0^T K_T(t,s)\,\d W_s.
$$
This equation can be interpreted as a stochastic differential equation with some anticipating Gaussian perturbation term $\int_0^T K_T(t,s)\,\d W_s$. Now the unique solution to \eqref{eq:SDE_langevin_integral} with an initial condition $X^\theta_0=0$ is given by
$$
X^\theta_t = \e^{-\theta t}\int_0^t \e^{\theta s}\d X_s.
$$
By using integration by parts and by applying the Fredholm representation of $X$ this can be written as 
$$
X^\theta_t = \int_0^T K_T(t,u)\,\d W_u - \theta \int_0^t \int_0^T \e^{-\theta t}\e^{\theta s}K_T(s,u)\,\d W_u \d s
$$
which, thanks to Stochastic Fubini's theorem, can be written as
$$
X^\theta_t = \int_0^T \left[K_T(t,u) - \theta \int_0^t \e^{-\theta t}\e^{\theta s}K_T(s,u)\, \d s\right]\d W_u.
$$
In other words, the solution $X^\theta$ is a Gaussian process with a Kernel
$$
K^\theta_T(t,u) = K_T(t,u) - \theta \int_0^t e^{-\theta (t-s)}K_T(s,u)\,\d s.
$$
Note that this is just an example how transfer principle can be applied in order to study stochastic differential equations. Indeed, for a more general equation
$$
\d X^a_t = a(t,X^a_t)\d t + \d X_t
$$
the existence or uniqueness result transfers immediately to the existence or uniqueness result of equation
$$
X^a_t = X^a_0 + \int_0^t a(s,X^a_s)\,\d s + \int_0^T K_T(t,u)\,\d W_u,
$$
and vice versa. 

Let us end this section by discussing briefly how the transfer principle can be used to build maximum likelihood estimators (MLE's) for the mean-reversal-parameter $\theta$ in equation \eqref{eq:SDE_langevin_integral}. For details on parameter estimation in such equations with general stationary-increment Gaussian noise we refer to \cite{sot-vii-langevin} and references therein. Let us assume that the noise $X$ in \eqref{eq:SDE_langevin_integral} is infinite-generate, in the sense that the Brownian motion $W$ in its Fredholm representation is a linear transformation of $X$.  Assume further that the transformation admits a kernel so that we can write
$$
W_t = \int_0^T K^{-1}_T(t,s)\, \d X_s.
$$
Then, by operating with the kernels $K_T$ and $K_T^{-1}$, we see that the equation \eqref{eq:SDE_langevin_integral} is equivalent to the anticipating equation
\begin{equation}\label{eq:anticipating-girsanov}
W^\theta_t = A_{T,t}(W^\theta) + W_t,
\end{equation} 
where
$$
A_{T,t}(W^\theta)
= - \theta\int_0^T\int_0^T K^{-1}_T(t,s)K_T(s,u)\, \d W^\theta_u \,\d s.
$$
Consequently, the MLE for the equation \eqref{eq:SDE_langevin_integral} is the MLE for the equation \eqref{eq:anticipating-girsanov}, which in turn can be constructed by using a suitable anticipating Girsanov theorem. There is a vast literature on how to do this, see e.g., \cite{Bishwal-2010} and references therein.

%%%%%%%%%%%%%%%%%%%%%%%%%%%%%%%%%%%%%%%%%%%%%%%%%%%%%%%%%%%%%%%%%%%%%%%%%%%%%%%
\section{Conclusions}

We have shown that virtually every Gaussian process admits a Fredholm representation with respect to a Brownian motion.  This apparently a simple fact has, as far a we know, remained unnoticed until now.  The Fredholm representation immediately yields the transfer principle with allows one to transfer the stochastic analysis of virtually any Gaussian process into stochastic analysis of the Brownian motion.  We have show how this can be done.  Finally, we have illustrated the power of the Fredholm representation and the associated transfer principle in many applications.

Stochastic analysis becomes easy with the Fredholm representation.  The only obvious problem is to construct the Fredholm kernel from the covariance function.  In principle this can be done algorithmically, but analytically it is very difficult.  The opposite construction is, however, trivial.  Therefore, if one begins the modeling with the Fredholm kernel and not with the covariance, one's analysis will be simpler and much more convenient. 

%%%%%%%%%%%%%%%%%%%%%%%%%%%%%%%%%%%%%%%%%%%%%%%%%%%%%%%%%%%%%%%%%%%%%%%%%%%%%%%
\bibliographystyle{siam}
\bibliography{bibli}
\end{document}